\documentclass[11pt, twoside, leqno]{article}
\usepackage{amssymb,latexsym}
\usepackage{enumerate}
\usepackage{color}
\usepackage{amsmath, amsthm, amsfonts, amssymb, mathrsfs}
\newtheorem{theorem}{Theorem}[section]
\newtheorem{corollary}[theorem]{Corollary}
\newtheorem{lemma}[theorem]{Lemma}
\newtheorem{proposition}[theorem]{Proposition}

\newtheorem{remark}[theorem]{Remark}

\numberwithin{equation}{section}

\numberwithin{equation}{section}

\begin{document}

\baselineskip=17pt

\title{Variational inequalities for bilinear averages}

\date{}

\maketitle

 \begin{center}
{\bf Honghai Liu}\\
School of Mathematics and Information Science,\\
Henan Polytechnic University,\\
Jiaozuo, Henan, 454003,  China \\
E-mail: {\it hhliu@hpu.edu.cn}\vskip 0.5cm
\end{center}

\renewcommand{\thefootnote}{}

\footnote{2010 \emph{Mathematics Subject Classification}: Primary
42B25; Secondary 42B20.}

\footnote{\emph{Key words and phrases}:
Variational inequalities, Averages, Bilinear operators,
Bilinear Littlewood-Paley operators}


\renewcommand{\thefootnote}{\arabic{footnote}}
\setcounter{footnote}{0}

\newpage

\begin{abstract}
We obtain variational inequalities for some classes of bilinear averages of one variable, generalizing the variational inequalities for averages of R. Jones {\it et al}. As an application we get almost everywhere convergence for the ergodic averages along cubes on a dynamical system.
\end{abstract}

 \section{Introduction}\label{sect1}

 The variational inequalities have been the subject of many
 recent articles in probability, ergodic theory and harmonic
 analysis.For linear version, the first variational inequality was proved by L\'epingle \cite{Lep76} for martingales (see \cite{PiXu88} for a simple proof). Bourgain \cite{Bou89} used L\'epingle's result to obtain corresponding variational estimates for the Birkhoff ergodic averages and then directly deduce pointwise convergence results without previous knowledge that pointwise convergence holds for a dense subclass of functions, which is quite diffcult in some ergodic models. A few years later, Jones and his collaborators systematically studied variational inequalities for ergodic averages in \cite{JKRW98}, \cite{JRW03}, \cite{CJRW2000} and \cite{CJRW2002}, see also \cite{HoMa,MTX,HL17}. Recently, several results on variational inequalities for discrete averaging operators of Radon type have also been established (cf. e.g. \cite{Kra14}, \cite{MiTr14}, \cite{MST15}, \cite{MTZ14}, \cite{Zor14}).
 \par
In this paper we concern with variational inequalities for some classes of bilinear averages, and their application to ergodic theory. In fact, the problem of almost everywhere convergence of multilinear ergodic averages plays an important role in ergodic theory. For instance, Demeter {\it et al}\ \cite{DTH08}
 considered the following multilinear averages and related ergodic averages:
 \begin{align}\label{RGaverage}
 T_{A,\mathbb R,r}(f_1,\cdots,f_{n-1})(x)=\frac1{(2r)^m}\int_{|t_1|,\cdots,|t_{m}|\le r}\prod_{i=1}^{n-1}f_i\big(x+\sum_{j=1}^ma_{i,j}t_j\big)d\vec{t},
 \end{align}
 and
  \begin{align}\label{VGaverage}
 T_{A,X,L}(f_1,\cdots,f_{n-1})(x)=\frac1{(2L+1)^m}\sum_{|l_1|,\cdots,|l_{m}|\le L}\prod_{i=1}^{n-1}f_i\big(S^{\sum_{j=1}^ma_{i,j}l_j}x\big),
 \end{align}
  where $n>1$, $m\ge1$, $A=(a_{i,j})$ is a $(n-1)\times m$ integer-valued matrix and $(X,\Sigma,m,S)$ is a dynamical system. This kind of averages are related to the Furstenberg recurrence theorem \cite{F77} and to Szemer\'{e}di's theorem \cite{S75} on arithmetic progressions, and are also connected to the result in \cite{GT08} that primes contain arbitrarily long progressions.To get the convergence, authors  established the almost everywhere convergence for $T_{A,X,L}$ for $f_1,\cdots,f_{n-1}\in L^\infty(X)$, proved $\sup\limits_{L>0}|T_{A,X,L}|$ maps $L^{p_1}(X)\times\cdots \times L^{p_{n-1}}(X)$ to $L^p(X)$ and extended the convergence result to the case when $ f_i\in L^{p_i}(X)$. The boundedness of $\sup\limits_{L>0}|T_{A,X,L}|$  is a consequence of an analogous boundedness for $\sup\limits_{r>0}|T_{A,\mathbb R,r}|$, because of transference arguments. But the problem of almost everywhere convergence of $T_{A,X,L}$ for $f_1,\cdots,f_{n-1}\in L^\infty(X)$ is quite difficult except some special case. An alternate method would be to prove variational inequalities for $T_{A,X,L}$ in $L$ without consider the almost everywhere convergence of $T_{A,X,L}$ for $f_1,\cdots,f_{n-1}\in L^\infty(X)$. In 2008, Demeter {\it et al}\ \cite{DLTH08} established an oscillation result(a weak variational inequality) which is used to prove the convergence for the signed average analog of Bourgain's return times theorem, and to provide a separate proof of Bourgain's theorem.
\par
Precisely, we primarily consider the almost everywhere convergence of the following bilinear averages:
\begin{align*}
Q_t(f,g)(x)=\frac1{t^2}\int_{|y|\le\frac t2}\int_{|z|\le\frac t2}f(x-y)g(x-z)dydz,
\end{align*}
where $t>0$, and $f,g$ are arbitrary measurable functions on $\mathbb R$. Note that averages $Q_t(f,g)$ are special cases of multilinear averages defined in \eqref{RGaverage} when $n=3$, $m=2$ and $A=I_{2\times2}$. We denote the family $\{Q_t(f,g)\}_{t>0}$ by $\mathcal Q(f,g)$. Before we can get into more details we need some definitions.
\par
For sequence $\{a_n\}$ and $\rho\ge1$ define the variational norm $V_\rho$ by
$$
\|\{a_n\}\|_{V_\rho}=\sup_{\{n_i\}}\big(\sum_{i}|a_{n_i}-a_{n_{i+1}}|^\rho\big)^{1/\rho},
$$
where the supremum is taken over all systems of indices $n_1<n_2<\cdots$.
Given an interval $I\in (0,\infty)$ and a family of complex numbers $\mathfrak a=\{a_t\}_{t\in I}$, the variational norm of the family $\mathfrak a$ is defined as
\begin{equation*}\label{q-ver number family}
\|\mathfrak a\|_{V_\rho(I)}=\sup\big(\sum_{i\geq1}
|a_{t_i}-a_{t_{i+1}}|^\rho\big)^{\frac{1}{\rho}},
\end{equation*}
where the supremum runs over all increasing sequences $\{t_i\in I:i\geq1\}$. It is trivial that
\begin{equation}\label{number contr ineq}
\|\mathfrak a\|_{L^\infty(I)}:=\sup_{t\in I}|a_t|
\le|a_{t_0}|+\|\mathfrak a\|_{V_\rho(I)}\quad\text{for\ any}\ t_0\in I\ \text{and}\ \rho\ge1.
\end{equation}
If $I=(0,\infty)$, we denote the variational norm $V_\rho(I)$ by $V_\rho$ for short.
\par
Given a family of Lebesgue measurable functions $\mathcal F=\{F_t(x)\}_{t>0}$ defined on $\mathbb{R}$, for fixed $x$ in $\mathbb{R}$ the value of the strong $\rho$-variation operator $V_\rho(\mathcal F)$ of the family $\mathcal F$ at $x$ is defined by
\begin{equation}\label{defini of Vq(F)}
V_\rho(\mathcal F)(x)=\|\{F_t(x)\}_{t>0}\|_{V_\rho},\quad \rho\ge1.
\end{equation}
It is easy to observe from the definition of $\rho$-variation norm that for fixed $x$ if $V_\rho(\mathcal F)(x)<\infty$, then $\{F_t(x)\}_{t>0}$ converges when $t\rightarrow0$ and $t\rightarrow\infty$. In particular, if $V_\rho(\mathcal F)$ belongs to some function spaces such as $L^p$ or $L^{p,\infty}$, then the family $\{F_t(x)\}_{t>0}$ converges almost everywhere without any additional condition. This is why mapping property of strong $\rho$-variation operator is so interesting in ergodic theory and harmonic analysis.
\par
The following theorem is a variational inequality for bilinear averages over cubes.
\begin{theorem}\label{var of the Q C}
For $\rho>2$, $1<p_1,p_2<\infty$ and $\frac1p=\frac1{p_1}+\frac1{p_2}$, we have
\begin{align*}
\|V_\rho(\mathcal{Q}(f,g))\|_{L^p(\mathbb R)}\le C\|f\|_{L^{p_1}(\mathbb R)}\|g\|_{L^{p_2}(\mathbb R)}.
\end{align*}
\end{theorem}
In addition to averages $\{Q_t(f,g)\}_{t>0}$, we introduce averages $\{\mathrm{Q}_L(\phi,\psi)\}_{L\in\mathbb N}$ defined on $\phi,\psi:\mathbb Z\rightarrow \mathbb R$ of compact support:
\begin{align*}
 \mathrm{Q}_L(\phi,\psi)(i)=\frac1{(2L+1)^2}\sum_{|l|,|k|\le L}\phi(i-l)\psi(i-k).
 \end{align*}
 The family of discrete averages $\{\mathrm{Q}_L(\phi,\psi)\}_{L\in\mathbb N}$ is denoted by $\mathbf{Q}(\phi,\psi)$. Moreover, we obtain the discrete version of Theorem \ref{var of the Q C} as follows.
\begin{corollary}\label{var of the Q D}
For $\rho>2$, $1<p_1,p_2<\infty$ and $\frac1p=\frac1{p_1}+\frac1{p_2}$, we have
\begin{align*}
\|V_\rho(\mathbf{Q}(\phi,\psi))\|_{L^p(\mathbb Z)}\le C\|\phi\|_{L^{p_1}(\mathbb Z)}\|\psi\|_{L^{p_2}(\mathbb Z)}.
\end{align*}
\end{corollary}

Let $(X,\Sigma,m,S)$ denote a dynamical system with $(X,\Sigma, m)$ a complete probability space and $S$ an invertible bimeasurable transformation such that $mS^{-1}=m$.
The closely related ergodic averages are given by
$$
\mathfrak{Q}_L(f,g)(x)=\frac1{(2L+1)^2}\sum_{|l_1|,|l_{2}|\le L}f\big(S^{l_1}x\big)g\big(S^{l_2}x\big).
$$
The sequence $\{\mathfrak{Q}_L(f,g)\}_{L}$ is denoted by $\mathscr Q(f,g)$.
Appealing to Corollary \ref{var of the Q D} and standard transfer methods like in \cite{DTH08,D07}, we get
\begin{corollary}\label{var of the ergodic Q C}
For $\rho>2$, $1<p_1,p_2<\infty$ and $\frac1p=\frac1{p_1}+\frac1{p_2}$, we have
\begin{align*}
\|V_\rho(\mathscr{Q}(f,g))\|_{L^p(X)}\le C\|f\|_{L^{p_1}(X)}\|g\|_{L^{p_2}(X)}.
\end{align*}
Moreover, for every dynamical system $(X,\Sigma,m,S)$, the averages over squares
$$
\frac1{(2N+1)^2}\sum_{i=-N}^{N}\sum_{j=-N}^{N}f\big(S^ix\big)g\big(S^jx\big)
$$
converge a.e. for $f\in L^{p_1}(X)$ and $g\in L^{p_2}(X)$.
\end{corollary}
\par
For $j,m\in\mathbb Z$, the dyadic interval in $\mathbb R$ is an interval of the form $[m2^j.(m+1)2^j)$. The set of all dyadic intervals with side-length $2^j$ is denoted by $\mathcal D_j$. The conditional expectation of a local integrable $f$ with respect to the increasing family of $\sigma-$algebras $\sigma(\mathcal D_j)$ generated by $\mathcal D_j$ is given by
$$
\mathbb E_jf(x)=\sum_{I\in \mathcal D_j}\frac1{|I|}\int_{I}f(y)dy\cdot\chi_{I}(x)
$$
for all $j\in\mathbb Z$. In view of the Lebesgue differentiation theorem, we have that
$$
\lim_{j\rightarrow \infty}\mathbb E_jf\rightarrow f,\ a.e.
$$
for $f\in L^2(\mathbb R)$. $\{\mathbb E_jf\}_j$ can be looked as a family of averages which are constructed from $f$ by certain averaging process. Moreover, there is a close connection between the martingale sequence $\{\mathbb E_jf\}_j$ and averages over cubes \cite{JKRW98,JRW03,JSW08}. Therefore, we consider the bilinear conditional expectation of two local integrable $f$ and $g$, which is given by
$$
\mathbb E_j(f,g)(x)=\sum_{I,J\in \mathcal D_j}\frac1{|I\times J|}\int_{I\times J}f(y)g(z)dydz\cdot\chi_{I\times J}(x,x).
$$
For the bilinear conditional expectation, we obtain the following variational inequality.
\begin{theorem}\label{var of Mar}
For $\rho>2$, $1<p_1,p_2<\infty$ and $\frac1p=\frac1{p_1}+\frac1{p_2}$, we have
\begin{align*}
\|V_\rho(\{\mathbb E_j(f,g)\}_j)\|_{L^p(\mathbb R)}\le C\|f\|_{L^{p_1}(\mathbb R)}\|g\|_{L^{p_2}(\mathbb R)}.
\end{align*}
\end{theorem}
Other family of bilinear averages are carried out by a suitable "approximation of the identity" as follows. Fix $\phi\in \mathscr{S}(\mathbb R^2)$ with $\int_{\mathbb R^2}\phi(x)dx =1$. For $t>0$, set $\phi_t(x,y)=t^{-2}\phi(x/t,y/t)$. The bilinear convolution operators are given by
\begin{equation*}
\phi_t(f,g)(x)=\int_{\mathbb R^2}\phi_t(x-y,x-z)f(y)g(z)dydz.
\end{equation*}
We denote $\{\phi_t(f,g)\}_{t>0}$ by $\Phi(f,g)$. In this setting we obtain the variational estimate as follows.
 \begin{theorem}\label{var of the 1.1}
For $\rho>2$, $1<p_1,p_2<\infty$ and $\frac1p=\frac1{p_1}+\frac1{p_2}$, we have
\begin{align*}
\|V_\rho(\Phi(f,g))\|_{L^p(\mathbb R)}\le C\|f\|_{L^{p_1}(\mathbb R)}\|g\|_{L^{p_2}(\mathbb R)}.
\end{align*}
\end{theorem}
In the next section we give the proof of the variational inequality for averages over cubes, which is a consequence of an vector-valued bilinear interpolation and an endpoint estimate for certain vector-valued operator. The discrete analogue is proved at the end of this section. The variational inequality for conditional expectations is treated in the same way in section 3. In final section we prove the variational estimate for approximations of the identity in the similar way. But, the $L^{p_1}\times L^{p_2}\rightarrow L^p$ bounds for all $1<p,p_1,p_2<\infty$ with $\frac1{p_1}+\frac1{p_2}=\frac1p$ and endpoint estimate can not be established directly, since those kernels are not  multiplicatively separable. We apply bilinear vector-valued Calder\'{o}n-Zygmund theory to deal with those problems.
\section{Variational inequality for averages over cubes}
In order to prove Theorem \ref{var of the Q C}, we present an $\mathcal B$-valued bilinear interpolation, where $\mathcal B$ is a Banach space, see \cite{JH12} and \cite{RS69}.
\begin{lemma}\label{B interpolation}
Suppose  and $T$ is a bilinear $\mathcal B$-valued operator. If $T$ is bounded from $L^{p_1}(\mathbb R)\times L^{p_2}(\mathbb R)$ into $L^{p,\infty}(\mathcal B)$ for all $1<p,p_1,p_2<\infty$ with $\frac1{p_1}+\frac1{p_2}=\frac1p$ and from $L^1(\mathbb R)\times L^1(\mathbb R)$ into $L^{1/2,\infty}(\mathcal B)$, then $T$ is bounded from $L^{p_1}(\mathbb R)\times L^{p_2}(\mathbb R)$ into $L^{p,\infty}(\mathcal B)$ for all $1<p_1,p_2<\infty$ with $\frac1{p_1}+\frac1{p_2}=\frac1p$.
\end{lemma}
We take the Banach space $\mathcal B=\{a(t):\|a\|_{\mathcal B}=\|a\|_{V_\rho}<\infty\}$.
Then, $V_\rho(\mathcal{Q}(f,g))(x)=\|\{Q_t(f,g)(x)\}_{t>0}\|_{\mathcal B}$.
Lemma \ref{B interpolation} implies Theorem \ref{var of Mar} is a consequence of the following two propositions.
\begin{proposition}\label{var of the Q P}
For $\rho>2$, $1<p,p_1,p_2<\infty$ and $\frac1p=\frac1{p_1}+\frac1{p_2}$, we have
\begin{align*}
\|V_\rho(\mathcal{Q}(f,g))\|_{L^p(\mathbb R)}\le C\|f\|_{L^{p_1}(\mathbb R)}\|g\|_{L^{p_2}(\mathbb R)}.
\end{align*}
\end{proposition}
\begin{proof} Similarly, we get
\begin{align*}
V_\rho(\mathcal{Q}(f,g))(x)&=V_\rho(\mathcal M(f)\cdot \mathcal M(g))(x)\\
&\le M(f)(x)\cdot V_\rho(\mathcal M(g))(x)+M(g)(x)\cdot V_\rho(\mathcal M(f))(x).
\end{align*}
By using H\"{o}lder's inequality and the variational inequalities for averages\cite{JKRW98,CJRW2000}, we get the desired result.
\end{proof}
\begin{lemma}\label{weak var of bilinear Q}
For $\rho>2$, we have
\begin{align}\label{Uf w estimate}
\lambda|\{x\in\mathbb R:V_\rho(\mathcal{Q}(f_1,f_2))(x)>\lambda\}|^2\le C\|f_1\|_{L^1(\mathbb R)}\|f_2\|_{L^1(\mathbb R)}
\end{align}
uniformly in $\lambda>0$.
\end{lemma}
\begin{proof}
 By scaling, we can assume that $\lambda=1$. Suppose that $f_1,f_2$ are step functions given by a finite linear combination of characteristic functions of disjoint dyadic intervals. In proving above weak endpoint type estimate, we may assume that
$$
\|f_1\|_{L^1}=\|f_2\|_{L^1}=1.
$$
The general case follows immediately by scaling. It suffices to prove
\begin{align*}
|\{x\in\mathbb R:V_\rho(\mathcal Q(f_1,f_2))(x)>1\}|\le C.
\end{align*}
We apply the Calder\'{o}n-Zygmund decomposition to functions $f_i$ at height $1$ to obtain functions $g_i$, $b_i$ and finite families dyadic intervals $\{I_{i,k}\}_k$ with disjoint interiors such that
$$
f_i=g_i+b_i\ \ \text{and}\ \ b_i=\sum_kb_{i,k}.
$$
For $i=1,2$, we have
$$
\text{support}(b_{i,k})\subseteq I_{i,k}
$$
$$
\int_{I_{i,k}}b_{i,k}(x)dx=0
$$
$$
\int_{I_{i,k}}|b_{i,k}(x)|dx\le C|I_{i,k}|
$$
$$
|\cup_kI_{i,k}|\le C
$$
$$
\|g_i\|_{L^1}\le \|f_i\|_{L^1}=1
$$
$$
\|g_i\|_{L^\infty}\le 2.
$$
For interval $I$, $\tilde{I}$ denotes the interval that is concentric with $I$ and has length $3|I|$. For convenience, we denote $\cup_k\tilde{I}_{1,k}$ and $\cup_i\tilde{I}_{2,i}$ by $\Omega_1$ and $\Omega_2$, respectively. Since
\begin{align*}
|\{x\in\mathbb R:V_\rho(\mathcal Q(f_1,f_2))(x)>1\}|&\le |\{x\in\mathbb R:V_\rho(\mathcal Q(f_1,f_2))(x)>1/4\}|+|\Omega_1|\\
&+|\Omega_2|+|\{x\notin\Omega_1:V_\rho(\mathcal Q(b_1,g_2))(x)>1/4\}|\\
&+|\{x\notin\Omega_2:V_\rho(\mathcal Q(g_1,b_2))(x)>1/4\}|\\
&+|\{x\notin\Omega_1\cup\Omega_2:V_\rho(\mathcal Q(b_1,b_2))(x)>1/4\}|,
\end{align*}
it suffices to estimate each of above six sets. Let us start with the first one. Applying Proposition \ref{var of the Q P}, we observe
\begin{align*}
|\{x\in\mathbb R:V_\rho(\mathcal Q(g_1,g_2))(x)>1/4\}|&\le C\|V_\rho(\mathcal Q(g_1,g_2))\|_{L^1}\le C\|g_1\|_{L^2}\|g_2\|_{L^2}\le C.
\end{align*}
Obviously,$|\Omega_1|+|\Omega_2|\le C$.
Now we turn to the fourth term. For $x\notin\Omega_1$ and $t\in(0,\infty)$, there are at most two $k$'s for which
\begin{equation*}
\frac1{t}\int_{x-\frac t2}^{x+\frac t2}b_{1,k}(y)dy\neq0.
\end{equation*}
 Indeed, it happens only if $I_{1,k}$ contains the starting point or endpoint of $(x-\frac t2,x+\frac t2)$. Hence,
 \begin{align*}
 V_\rho(\mathcal Q(b_1,g_2))(x)&= \sup_{\{t_j\}\searrow0}\bigg(\sum_j\big|\sum_k[M_{t_j}(b_{1,k},g_2)(x)-M_{t_{j+1}}(b_{1,k},g_2)(x)]\big|^\rho\bigg)^{1/\rho}\\
 &\le C\sup_{\{t_j\}\searrow0}\bigg(\sum_j\sum_k|M_{t_j}(b_{1,k},g_2)(x)-M_{t_{j+1}}(b_{1,k},g_2)(x)|^\rho\bigg)^{1/\rho}\\
 &\le C\bigg(\sum_kV_\rho(\mathcal Q(b_{1,k},g_2))^\rho(x)\bigg)^{1/\rho}.
\end{align*}
For $x\notin\tilde{I}_{1,k}$, we assume $x$ is on the right of $I_{1,k}$, the other case can be treated in the same way. We can choose a monotone decreasing sequence $\{t_j(x)\}_j$ approaching $0$ such that
\begin{align*}
V_\rho(\mathcal Q(b_{1,k},g_2))(x)&\le C\sum_j\big|Q_{t_j(x)}(b_{1,k},g_2)(x)-Q_{t_{j+1}(x)}(b_{1,k},g_2)(x)\big|\\
&\lesssim |Q_{t_{j_0}(x)}(b_{1,k},g_2)(x)|+\sum_{j=j_0}^{j_1-1}\big|Q_{t_j(x)}(b_{1,k},g_2)(x)-Q_{t_{j+1}(x)}(b_{1,k},g_2)(x)\big|\\
&+|Q_{t_{j_1}(x)}(b_{1,k},g_2)(x)|\\
&\lesssim \frac1{t_{j_1}(x)}\|b_{1,k}\|_{L^1}+\sum_{j=j_0}^{j_1-1}\big|M_{t_j(x)}(b_{1,k})(x)-M_{t_{j+1}(x)}(b_{1,k})(x)\big|\\
&+\sum_{j=j_0}^{j_1-1}\big|M_{t_j(x)}(g_2)(x)-M_{t_{j+1}(x)}(g_2)(x)\big||M_{t_{j+1}(x)}(b_{1,k})(x)|,
\end{align*}
where $x-t_{j_0}(x)\in I_{1,k}$ and $x-t_{j_0-1}(x)\notin I_{1,k}$, $x-t_{j_1}(x)\in I_{1,k}$ and $x-t_{j_1+1}(x)+x\notin I_{1,k}$, and we have used the fact that $\|M(g_2)\|_{L^\infty}\le 2$. Clearly, $t_{j_1}(x)\sim d(x,I_{1,k})$ for $x\notin I_{1,k}$. Then, the second summand is dominated by
\begin{align*}
&\sum_{j=j_0}^{j_1-1}\big|\frac1{t_j(x)}-\frac1{t_{j+1}(x)}\big|\|b_{1,k}\|_{L^1}+\sum_{j=j_0}^{j_1-1}\frac1{t_{j+1}(x)}\big|\int_{x-\frac{t_j(x)}2}^{x+\frac{t_j(x)}2}b_{1,k}(y)dy-\int_{x-\frac{t_{j+1}(x)}2}^{x+\frac{t_{j+1}(x)}2}b_{1,k}(y)dy\big|\\
&\lesssim \frac1{t_{j_1}(x)}\|b_{1,k}\|_{L^1}\le \frac {C\|b_{1,k}\|_{L^1}}{d(x,I_{1,k})}.
\end{align*}
For the third summand, it is controlled by
\begin{align*}
&\sum_{j=j_0}^{j_1-1}\big|\frac1{t_j(x)}-\frac1{t_{j+1}(x)}\big|\frac{t_{j_0}(x)}{t_{j_1}(x)}\|b_{1,k}\|_{L^1}+\sum_{j=j_0}^{j_1-1}\frac1{t_{j+1}(x)}\big|\int_{x-\frac{t_j(x)}2}^{x+\frac{t_j(x)}2}g_2(z)dz-\int_{x-\frac{t_{j+1}(x)}2}^{x+\frac{t_{j+1}(x)}2}g_2(z)dz\big|\frac{\|b_{1,k}\|_{L^1}}{t_{j_1}(x)}\\
&\lesssim \frac{t_{j_0}(x)}{t^2_{j_1}(x)}\|b_{1,k}\|_{L^1}\lesssim \frac{d(x,I_{1,k})+|I_{1,k}|}{d(x,I_{1,k})^2}\|b_{1,k}\|_{L^1},
\end{align*}
where we used the fact $\|g_2\|_{L^\infty}\le2$ and $\|g_2\|_{L^1}\le 1$.
\par
As a result, we get
\begin{align*}
 \big|\{x\notin\Omega_1:V_\rho(\mathcal Q(b_1,g_2))(x)>1/4\}\big|&\le C\sum_k\int_{(\tilde{I}_{1,k})^c}V_\rho(\mathcal Q(b_{1,k},g_2))^\rho(x)dx\\
 &\le C\sum_k\|b_{1,k}\|_{L^1}^\rho\int_{(\tilde{I}_{1,k})^c}\frac{(d(x,I_{1,k})+|I_{1,k}|)^\rho}{d(x,I_{1,k})^{2\rho}}dx\\
 &\le C\sum_k\|b_{1,k}\|_{L^1}^\rho |I_{1,k}|^{1-\rho}\le C\sum_k|I_{1,k}|\le C.
\end{align*}
 The fifth term can be treated in the similar way, we obtain
 \begin{align*}
 \big|\{x\notin\Omega_2:V_\rho(\mathcal Q(g_1,b_2))(x)>1/4\}\big|\le C.
\end{align*}
For the last one, we write
\begin{align*}
b_1(y)b_2(z)&=\sum_kb_{1,k}(y)\sum_{i:|I_{2,i}|\le|I_{1,k}|}b_{2,i}(z)+\sum_ib_{2,i}(z)\sum_{k:|I_{1,k}|\le|I_{2,i}|}b_{1,k}(y)\\
&:=\sum_kb_{1,k}(y)b_2^{(k)}(z)+\sum_ib_{2,i}(z)b_1^{(i)}(y).
\end{align*}
Then, for $x\notin\Omega_1\cup\Omega_2$, we observe that
\begin{align*}
 V_\rho(\mathcal Q(b_1,b_2))(x)\le \bigg(\sum_kV_\rho(\mathcal Q(b_{1,k},b_2^{(k)}))^\rho(x)\bigg)^{1/\rho}+\bigg(\sum_iV_\rho(\mathcal Q(b_1^{(i)},b_{2,i}))^\rho(x)\bigg)^{1/\rho},
\end{align*}
where we use the fact that for $x\notin\Omega_1\cup\Omega_2$ and $t\in(0,\infty)$, there are at most two $k$'s and two $i$'s  for which
\begin{equation*}
\frac1{t}\int_{x-\frac t2}^{x+\frac t2}b_{1,k}(y)dy\neq0\ \ \text{and}\ \ \frac1{t}\int_{x-\frac t2}^{x+\frac t2}b_{2,i}(z)dz\neq0.
\end{equation*}
Hence, we see that
\begin{align*}
|\{x\notin\Omega_1\cup \Omega_2:V_\rho(\mathcal Q(b_1,b_2))(x)>1/4\}|&\le |\{x\notin\Omega_1\cup \Omega_2:\big(\sum_kV_\rho(\mathcal Q(b_{1,k},b_2^{(k)}))^\rho\big)^{\frac1\rho}(x)>1/8\}|\\
&+|\{x\notin\Omega_1\cup \Omega_2:\big(\sum_iV_\rho(\mathcal Q(b_1^{(i)},b_{2,i}))^\rho\big)^{\frac1\rho}(x)>1/8\}|.
\end{align*}
It suffices to consider the first term, the other one can be treated in the same way. For $x\notin\Omega_1\cup \Omega_2$ and $t>d(x,I_{1,k})$ such that $M_t(b_{1,k})(x)\neq0$, there are at most two summands $b_{2,i}$ in $b_2^{(k)}$ for which
\begin{equation*}
\int_{x-\frac t2}^{x+\frac t2}b_{2,i}(z)dz\neq0\ \ \text{and}\ \ \big|\int_{x-\frac t2}^{x+\frac t2}b_{2,i}(z)dz\big|\le |I_{2,i}|\le |I_{1,k}|.
\end{equation*}
Notice that dyadic intervals $\{I_{2,i}\}_i$ are with disjoint interiors. Moreover, for above $x$ and $t$, we obtain
\begin{align*}
|M_t(b_2^{(k)})|\le \frac{2|I_{1,k}|}{d(x,I_{1,k})}\le 2\ \ \text{and}\ \ |Q_t(b_{1,k},b_2^{(k)})|&\le \frac1t\|b_{1,k}\|_{L^1}M_t(b_2^{(k)})\le \frac2t\|b_{1,k}\|_{L^1}.
\end{align*}
For $x\notin I_{1,k}\cup \Omega_2$, we assume $x$ is on the right of $I_{1,k}$. We can choose a monotone decreasing sequence $\{t_j(x)\}_j$ approaching $0$ such that
\begin{align*}
V_\rho(\mathcal Q(b_{1,k},b_2^{(k)}))(x)&\le C\sum_j\big|Q_{t_j(x)}(b_{1,k},b_2^{(k)})(x)-Q_{t_{j+1}(x)}(b_{1,k},b_2^{(k)})(x)\big|\\
&\lesssim |Q_{t_{j_0}(x)}(b_{1,k},b_2^{(k)})(x)|+\sum_{j=j_0}^{j_1-1}\big|Q_{t_j(x)}(b_{1,k},b_2^{(k)})(x)-Q_{t_{j+1}(x)}(b_{1,k},b_2^{(k)})(x)\big|\\
&+|Q_{t_{j_1}(x)}(b_{1,k},b_2^{(k)})(x)|\\
&\lesssim\frac1{t_{j_1}(x)}\|b_{1,k}\|_{L^1}+\sum_{j=j_0}^{j_1-1}\big|M_{t_j(x)}(b_{1,k})(x)-M_{t_{j+1}(x)}(b_{1,k})(x)\big||M_{t_j(x)}(b_2^{(k)})(x)|\\
&+\sum_{j=j_0}^{j_1-1}\big|M_{t_j(x)}(b_2^{(k)})(x)-M_{t_{j+1}(x)}(b_2^{(k)})(x)\big||M_{t_{j+1}(x)}(b_{1,k})(x)|,
\end{align*}
where $x-t_{j_0}(x)\in I_{1,k}$ and $x-t_{j_0-1}(x)\notin I_{1,k}$, $x-t_{j_1}(x)\in I_{1,k}$ and $x-t_{j_1+1}(x)\notin I_{1,k}$. The second summand is dominated by
\begin{align*}
&\sum_{j=j_0}^{j_1-1}\big|\frac1{t_j(x)}-\frac1{t_{j+1}(x)}\big|\|b_{1,k}\|_{L^1}+\sum_{j=j_0}^{j_1-1}\frac1{t_{j+1}(x)}\big|\int_{x-\frac{t_j(x)}2}^{x+\frac{t_j(x)}2}b_{1,k}(y)dy-\int_{x-\frac{t_{j+1}(x)}2}^{x+\frac{t_{j+1}(x)}2}b_{1,k}(y)dy\big|\\
&\lesssim \frac1{t_{j_1}(x)}\|b_{1,k}\|_{L^1}\le \frac {C\|b_{1,k}\|_{L^1}}{d(x,I_{1,k})}.
\end{align*}
We estimate the third summand as
\begin{align*}
&\sum_{j=j_0}^{j_1-1}\big|\frac1{t_j(x)}-\frac1{t_{j+1}(x)}\big|\frac{|I_{1,k}|}{t_{j_1}(x)}\|b_{1,k}\|_{L^1}+\sum_{j=j_0}^{j_1-1}\frac{\|b_{1,k}\|_{L^1}}{t^2_{j_1}(x)}\big|\int_{x-\frac{t_j(x)}2}^{x+\frac{t_j(x)}2}b_2^{(k)}(z)dz-\int_{x-\frac{t_{j+1}(x)}2}^{x+\frac{t_{j+1}(x)}2}b_2^{(k)}(z)dz\big|\\
&\lesssim \frac{|I_{1,k}|}{t^2_{j_1}(x)}\|b_{1,k}\|_{L^1}\le \frac {C\|b_{1,k}\|_{L^1}}{d(x,I_{1,k})},
\end{align*}
where we use the fact $|I_{1,k}|\le t_{j_1}(x)$. Finally, using Chebyshev's inequality,
\begin{align*}
 \big|\{x\notin\Omega_1\cup \Omega_2:\big(\sum_kV_\rho(\mathcal Q(b_{1,k},b_2^{(k)}))^\rho\big)^{\frac1\rho}(x)>1/8\}\big|&\le C\sum_k\int_{(\tilde{I}_{1,k})^c}V_\rho(\mathcal Q(b_{1,k},b_2^{(k)}))^\rho(x)dx\\
 &\le C\sum_k\|b_{1,k}\|_{L^1}^\rho\int_{(\tilde{I}_{1,k})^c}\frac1{d(x,I_{1,k})^{\rho}}dx\\
 &\le C\sum_k\|b_{1,k}\|_{L^1}^\rho |I_{1,k}|^{1-\rho}\le C.
\end{align*}
This completes the proof of Proposition \ref{weak var of bilinear Q}.
\end{proof}
Now let turn to the proof of Corollary \ref{var of the Q D}.
\begin{proof}
For each $\phi,\psi:\mathbb Z\rightarrow\mathbb Z$ we consider functions like $f:\mathbb R\rightarrow \mathbb R$ with
\begin{equation*}
f(x)= \begin{cases}

   \phi([x]), &\mbox{$[x]+\frac14\le x\le[x]+\frac12$,}\\

   0, &\mbox{otherwise,}

   \end{cases}
\end{equation*}
and $g:\mathbb R\rightarrow \mathbb R$ with
\begin{equation*}
g(x)= \begin{cases}

   \psi([x]), &\mbox{$[x]+\frac14\le x\le[x]+\frac12$,}\\

   0, &\mbox{otherwise.}
\end{cases}
\end{equation*}
For $L\in\mathbb N$ and $i\in\mathbb Z$, we observe that
$$
\mathrm{Q}_L(\phi,\psi)(i)=4Q_{L+\frac12}(f,g)(x),\ \ x\in[i,i+\frac34].
$$
Further, we get that
$$
V_\rho\big(\mathbf{Q}(\phi,\psi)\big)(i)\le4V_\rho\big(\mathcal Q(f,g)\big)(x),\ \ x\in[i,i+\frac34].
$$
For the variational inequality for averages over cubes in Theorem \ref{var of the Q C} we deduce that
\begin{align*}
\|V_\rho\big(\mathbf{Q}(\phi,\psi)\big)\|_{l^p(\mathbb Z)}&=\bigg(\sum_i\big|V_\rho\big(\mathbf{Q}(\phi,\psi)\big)(i)\big|^p\bigg)^{1/p}\\
&\le 4\big(\frac43\big)^{1/p}\bigg(\sum_i\int_{i}^{i+3/4}\big|V_\rho\big(\mathcal Q(f,g)\big)(x)\big|^pdx\bigg)^{1/p}\\
&\le4\big(\frac43\big)^{1/p}\big\|V_\rho\big(\mathcal Q(f,g)\big)\big\|_{L^p(\mathbb R)}\le C\|f\|_{L^{p_1}(\mathbb R)}\|g\|_{L^{p_2}(\mathbb R)}\\
&\le C\|\phi\|_{l^{p_1}(\mathbb Z)}\|\psi\|_{l^{p_2}(\mathbb Z)}.
\end{align*}
\end{proof}
\section{Variational inequality for conditional expectations}
In the same way, we apply Lemma \ref{B interpolation} and take the Banach space $\mathcal B=\{a(j):\|a\|_{\mathcal B}=\|a\|_{V_\rho}<\infty\}$.
Then, $V_\rho(\{\mathbb E_j(f,g)\}_j)=\|\{\mathbb E_j(f,g)\}_j\|_{\mathcal B}$.
Lemma \ref{B interpolation} implies Theorem \ref{var of Mar} is a consequence of the following two propositions.
\begin{proposition}\label{var of bilinear con exp}
For $\rho>2$, $1<p,p_1,p_2<\infty$ and $\frac1p=\frac1{p_1}+\frac1{p_2}$, we have
\begin{align*}
\|V_\rho(\{\mathbb E_j(f,g)\}_j)\|_{L^p(\mathbb R)}\le C\|f\|_{L^{p_1}(\mathbb R)}\|g\|_{L^{p_2}(\mathbb R)}.
\end{align*}
\end{proposition}
\begin{proof}
Obviously, we have
$$
\mathbb E_j(f,g)(x)=\sum_{I,J\in \mathcal D_j}\frac1{|I|}\int_{I}f(y)dy\chi_{I}(x)\frac1{|J|}\int_{J}f(y)g(z)dz\chi_{J}(x)=\mathbb E_j(f)(x)\mathbb E_j(g)(x).
$$
Then, we get
\begin{align*}
|\mathbb E_{j_{n+1}}(f,g)-\mathbb E_{j_n}(f,g)|&=|\mathbb E_{j_{n+1}}(f)\mathbb E_{j_{n+1}}(g)-\mathbb E_{j_n}(f)\mathbb E_{j_n}(g)|\\
&\le |\mathbb E_{j_{n+1}}(f)-\mathbb E_{j_n}(f)|\cdot|\mathbb E_{j_{n+1}}(g)|+|\mathbb E_{j_{n+1}}(g)-\mathbb E_{j_n}(g)|\cdot|\mathbb E_{j_n}(f)|.
\end{align*}
By applying H\"{o}lder's inequality and L\'{e}pingle's inequality \cite{Lep76}, we obtain
\begin{align*}
\|V_\rho(\{\mathbb E_j(f,g)\}_j)\|_{L^p(\mathbb R)}&\le \|M(g)\cdot V_\rho(\{\mathbb E_j(f)\}_j)\|_{L^p}+\|M(f)\cdot V_\rho(\{\mathbb E_j(g)\}_j)\|_{L^p}\\
&\le C\|f\|_{L^{p_1}(\mathbb R)}\|g\|_{L^{p_2}(\mathbb R)}.
\end{align*}
This completes the proof of Proposition \ref{var of bilinear con exp}.
\end{proof}
\begin{remark}
In fact, above bilinear variational inequality holds for $p=1$, $1<p_1,p_2<\infty$ and $\frac1p=\frac1{p_1}+\frac1{p_2}$.
\end{remark}
The second proposition is the variational weak endpoint type estimate for conditional expectation sequence.
\begin{proposition}\label{weak var of bilinear con exp}
For $\rho>2$, we have
\begin{align*}
\lambda|\{x\in\mathbb R:V_\rho(\{\mathbb E_j(f_1,f_2)\}_j)(x)>\lambda\}|^2\le C\|f_1\|_{L^1(\mathbb R)}\|f_2\|_{L^1(\mathbb R)}
\end{align*}
uniformly in $\lambda>0$.
\end{proposition}
\begin{proof}
By scaling, we assume that $\lambda=1$ and $\|f_1\|_{L^1}=\|f_2\|_{L^1}=1$,
the general case follows immediately by scaling.
It suffices to prove
\begin{align*}
|\{x\in\mathbb R:V_\rho(\{\mathbb E_j(f_1,f_2)\}_j)(x)>1\}|\le  C.
\end{align*}
Analogously, we apply the Calder\'{o}n-Zygmund decomposition to functions $f_i$ at height $1$ to obtain functions $g_i$, $b_i$ and dyadic intervals $\{I_{i,k}\}_k$ such that
$f_i=g_i+b_i\ \ \text{and}\ \ b_i=\sum_kb_{i,k}$. Since
\begin{align*}
|\{x\in\mathbb R:V_\rho(\{\mathbb E_j(f_1,f_2)\}_j)(x)>1\}|&\le |\{x\in\mathbb R:V_\rho(\{\mathbb E_j(g_1,g_2)\}_j)(x)>1/4\}|+|\Omega_1|\\
&+|\Omega_2|+|\{x\notin\Omega_1:V_\rho(\{\mathbb E_j(b_1,g_2)\}_j)(x)>1/4\}|\\
&+|\{x\notin\Omega_2:V_\rho(\{\mathbb E_j(g_1,b_2)\}_j)(x)>1/4\}|\\
&+|\{x\notin\Omega_1\cup\Omega_2:V_\rho(\{\mathbb E_j(b_1,b_2)\}_j)(x)>1/4\}|,
\end{align*}
it suffices to estimate each of above six sets. Applying Proposition \ref{var of bilinear con exp}, we observe
\begin{align*}
|\{x\in\mathbb R:V_\rho(\{\mathbb E_j(g_1,g_2)\}_j)(x)>1/4\}|&\le C\|V_\rho(\{\mathbb E_j(g_1,g_2)\}_j)\|_{L^1}\le C\|g_1\|_{L^2}\|g_2\|_{L^2}\le C.
\end{align*}
Clearly,$|\Omega_1|+|\Omega_2|\le C$. Note that $\mathbb E_j(b_{1,k})(x)=0$ for $x\notin\tilde{I}_{1,k}$. Hence $\mathbb E_j(b_1,g_2)(x)=\mathbb E_j(b_1)(x)\cdot\mathbb E_j(g_2)(x)=0$ for $x\notin\Omega_1$. Consequently,
\begin{align*}
|\{x\notin\Omega_1:V_\rho(\{\mathbb E_j(b_1,g_2)\}_j)(x)>1/4\}|=|\{x\notin\Omega_1\cup\Omega_2:V_\rho(\{\mathbb E_j(b_1,b_2)\}_j)(x)>1/4\}|=0.
\end{align*}
Similarly,
\begin{align*}
|\{x\notin\Omega_2:V_\rho(\{\mathbb E_j(g_1,b_2)\}_j)(x)>1/4\}|=0.
\end{align*}
This proves Proposition \ref{weak var of bilinear con exp}.
\end{proof}

\section{Variational inequality for approximations of the identity}

In order to prove Theorem \ref{var of the 1.1}, we view the kernel $\{\phi_t(y,z)\}_{t>0}$ as having values in the Banach space
\begin{align}\label{def of B}
\mathcal B=\{a(t):\|a\|_{\mathcal B}=\|a\|_{V_\rho}<\infty\}.
\end{align}
Then, $V_\rho(\Phi(f,g))(x)=\|\{\phi_t(f,g)(x)\}_{t>0}\|_{\mathcal B}$. Lemma \ref{B interpolation} implies Theorem \ref{var of the 1.1} is a consequence of the following two propositions:
 \begin{proposition}\label{var of app >1}
For $\rho>2$, $1<p,p_1,p_2<\infty$ and $\frac1p=\frac1{p_1}+\frac1{p_2}$, we have
\begin{align*}
\|V_\rho(\Phi(f,g))\|_{L^p(\mathbb R)}\le C\|f\|_{L^{p_1}(\mathbb R)}\|g\|_{L^{p_2}(\mathbb R)}.
\end{align*}
\end{proposition}
\begin{proposition}\label{app var weak}
For $\rho>2$, then
\begin{align*}
\lambda\big|\big\{x\in\mathbb R:V_\rho(\Phi(f,g))(x)>\lambda\big\}\big|^2\le C\|f\|_{L^1(\mathbb R)}\|g\|_{L^1(\mathbb R)}
\end{align*}
for any $\lambda>0$.
\end{proposition}
\subsection{Variational inequality with $1<p,p_1,p_2<\infty$.}
The goal of this subsection is to prove Proposition \ref{var of app >1}. Let $\varphi\in \mathscr{S}(\mathbb R)$ and $\int_{\mathbb R}\varphi(x)dx =1$. Then, we have the following pointwise estimate:
\begin{align*}
V_\rho(\Phi(f,g))\le V_\rho(\{\varphi_t(f)\cdot\varphi_t(g)\}_{t>0})+V_\rho(\{\phi_t(f,g)-\varphi_t(f)\cdot\varphi_t(g)\}_{t>0}).
\end{align*}
Hence, it suffices to estimate the $L^p$ norms of $V_\rho(\{\varphi_t(f)\cdot\varphi_t(g)\}_{t>0})$ and $V_\rho(\{\phi_t(f,g)-\varphi_t(f)\cdot\varphi_t(g)\}_{t>0})$.
\begin{lemma}\label{var of bilinear var}
For $\rho>2$, $1<p,p_1,p_2<\infty$ and $\frac1p=\frac1{p_1}+\frac1{p_2}$, we have
\begin{align*}
\|V_\rho(\{\varphi_t(f)\cdot\varphi_t(g)\}_{t>0})\|_{L^p(\mathbb R)}\le C\|f\|_{L^{p_1}(\mathbb R)}\|g\|_{L^{p_2}(\mathbb R)}.
\end{align*}
\end{lemma}
\begin{proof}
Note that
\begin{align*}
&|\varphi_{t_i}(f)\cdot\varphi_{t_i}(g)-\varphi_{t_{i+1}}(f)\cdot\varphi_{t_{i+1}}(g)|\\
\le& |\varphi_{t_i}(f)-\varphi_{t_{i+1}}(f)|\cdot|\varphi_{t_i}(g)|+|\varphi_{t_i}(g)-\varphi_{t_{i+1}}(g)|\cdot |\varphi_{t_{i+1}}(f)|.
\end{align*}
Then, by using H\"{o}lder's inequality and Theorem 2.6 in \cite{HL17}, we get
\begin{align*}
\|V_\rho(\{\varphi_t(f)\cdot\varphi_t(g)\}_{t>0})\|_{L^p(\mathbb R)}&\le \|M(g)\cdot V_\rho(\{\varphi_t(f)\}_{t>0})\|_{L^p}+\|M(f)\cdot V_\rho(\{\varphi_t(g)\}_{t>0})\|_{L^p}\\
&\le C\|f\|_{L^{p_1}(\mathbb R)}\|g\|_{L^{p_2}(\mathbb R)}.
\end{align*}
\end{proof}
The long variation operator $V^L_\rho(\mathcal F)$ of the family $\mathcal F$ at $x$ is defined by
\begin{equation}\label{defini of Vq(F)}
V^L_\rho(\mathcal F)(x)=\|\{F_{2^n}(x)\}_{n\in\mathbb{Z}}\|_{V_\rho},\quad \rho\ge1.
\end{equation}
Moreover, the short variation operator
 $$
S_2(\mathcal F)(x)=\bigg(\sum_{j\in\mathbb Z}\|\{F_t(x)\}_{t>0}\|_{V_2[2^j,2^{j+1}]}^2\bigg)^{1/2}.
 $$

Then the following pointwise comparison holds.
 \begin{lemma}\ {\rm (see \cite[Lemma 1.3]{JSW08})}\label{lem:convert lemma}
\begin{equation}\label{pcls}
V_\rho(\mathcal F)(x)\lesssim V^L_\rho(\mathcal F)(x)+S_2(\mathcal F)(x).
\end{equation}
 \end{lemma}
\begin{lemma}\label{var of App diff}
For $\rho>2$,$1<p,p_1,p_2<\infty$ and $\frac1p=\frac1{p_1}+\frac1{p_2}$, we have
\begin{align*}
\|V_\rho(\{\phi_t(f,g)-\varphi_t(f)\cdot\varphi_t(g)\}_{t>0})\|_{L^p(\mathbb R)}\le C\|f\|_{L^{p_1}(\mathbb R)}\|g\|_{L^{p_2}(\mathbb R)}.
\end{align*}
\end{lemma}
\begin{proof}
To estimate the $L^p$ norm of $V_\rho(\{\phi_t(f,g)-\varphi_t(f)\cdot\varphi_t(g)\}_{t>0})$,we denote the function $\phi(y,z)-\varphi(y)\varphi(z)$ by $\psi(y,z)$ for convenience. \eqref{pcls} reduces above desired estimate to
\begin{align}\label{var of App diff L}
\|V^L_\rho(\{\psi_t(f,g)\}_{t>0})\|_{L^p(\mathbb R)}\le C\|f\|_{L^{p_1}(\mathbb R)}\|g\|_{L^{p_2}(\mathbb R)}
\end{align}
and
\begin{align}\label{var of App diff S}
\|S_2(\{\psi_t(f,g)\}_{t>0})\|_{L^p(\mathbb R)}\le C\|f\|_{L^{p_1}(\mathbb R)}\|g\|_{L^{p_2}(\mathbb R)}.
\end{align}
\par
We show \eqref{var of App diff L} first. Clearly, for $\rho>2$ we have
\begin{align*}
V^L_\rho(\{\phi_t(f,g)-\varphi_t(f)\cdot\varphi_t(g)\}_{t>0})=V^L_\rho(\{\psi_t(f,g)\}_{t>0})\le \big(\sum_j|\psi_{2^j}(f,g)|^2\big)^{1/2}.
\end{align*}
Hence, it suffices to prove
\begin{align}\label{squ fun est}
\|\big(\sum_j|\psi_{2^j}(f,g)|^2\big)^{1/2}\|_{L^p(\mathbb R)}\le C\|f\|_{L^{p_1}(\mathbb R)}\|g\|_{L^{p_2}(\mathbb R)},
\end{align}
for $1<p,p_1,p_2<\infty$ and $\frac1p=\frac1{p_1}+\frac1{p_2}$.
\par
To obtain \eqref{squ fun est}, we apply \cite[Theorem 1.1]{JH12} and verify $\psi$ satisfying related conditions. Note that $\phi\in\mathscr S(\mathbb R^2)$ and $\varphi\in\mathscr S(\mathbb R)$, then $\psi\in\mathscr S(\mathbb R^2)$. Hence, for any $N\in\mathbb N$ and multi-indices $\alpha$ we have
\begin{align*}
|\partial^\alpha\psi(y,z)|\le \frac{C_N}{(1+|y|+|z|)^{2N}}\le \frac{C_N}{(1+|y|)^N(1+|z|)^N}.
\end{align*}
Moreover, it satisfies the cancellation condition
\begin{align*}
\int_{\mathbb R^2}\psi(y,z)dydz=\int_{\mathbb R^2}\phi(y,z)dydz-\int_{\mathbb R}\varphi(y)dy\cdot\int_{\mathbb R}\varphi(z)dz=0.
\end{align*}
\par
As a result, we obtain
\begin{align*}
\|V^L_\rho(\{\phi_t(f,g)-\varphi_t(f)\cdot\varphi_t(g)\}_{t>0})\|_{L^p(\mathbb R)}&\le \|\big(\sum_j|\psi_{2^j}(f,g)|^2\big)^{1/2}\|_{L^p(\mathbb R)}\\
&\le C\|f\|_{L^{p_1}(\mathbb R)}\|g\|_{L^{p_2}(\mathbb R)},
\end{align*}
and complete the proof of \eqref{var of App diff L}.
\par
Next we turn to proof of \eqref{var of App diff S}. By using Bergh and Peetre's \cite{BP} estimate
\begin{align*}
\|a\|_{V_\rho}\le \|a\|_{L^\rho}^{1/\rho'}\|a'\|_{L^\rho}^{1/\rho},
\end{align*}
we observe that
\begin{align*}
S_2^2(\{\psi_t(f,g)\}_{t>0})(x)&=\sum_k\|\{\psi_t(f,g)\}_{t>0}\|^2_{V_2[2^k,2^{k+1}]}\\
&\le \sum_k\|\psi_t(f,g)\|_{L_t^2[2^k,2^{k+1}]}\|\frac d{dt}\psi_t(f,g)\|_{L_t^2[2^k,2^{k+1}]}\\
&\le C\bigg(\int_0^\infty|\psi_t(f,g)(x)|^2\frac{dt}t\bigg)^{1/2}\bigg(\int_0^\infty|\tilde{\psi}_t(f,g)(x)|^2\frac{dt}t\bigg)^{1/2}\\
&:=C G(f,g)(x)\tilde{G}(f,g)(x),
\end{align*}
where $\tilde{\psi}(y,z)=2\psi(y,z)+y\partial_y\psi(y,z)+z\partial_z\psi(y,z)$. Note that $\psi,\tilde{\psi}\in \mathscr S(\mathbb R^2)$, for any $N\in\mathbb N$ we have
\begin{align*}
|\hat{\psi}(\xi,\eta)|+|\hat{\tilde\psi}(\xi,\eta)|\le \frac C{(1+|(\xi,\eta)|)^N}\ \ \text{and}\ \ \hat{\psi}(0,0)=\hat{\tilde\psi}(0,0)=0.
\end{align*}
Using \cite[Example 2.1]{SXY17} and \cite[Theorem 1.2]{XPY15}, we get
\begin{align*}
\|G(f,g)\|_{L^p(\mathbb R)}+\|\tilde G(f,g)\|_{L^p(\mathbb R)}\le C\|f\|_{L^{p_1}(\mathbb R)}\|g\|_{L^{p_2}(\mathbb R)}
\end{align*}
for $1<p,p_1,p_2<\infty$. Furthermore, by  H\"{o}lder's inequality
\begin{align*}
\|S_2(\{\psi_t(f,g)\}_{t>0})\|^p_{L^p(\mathbb R)}&=\int_{\mathbb R}S_2(\{\psi_t(f,g)\}_{t>0})^{2\cdot \frac p2}(x)dx\le \int_{\mathbb R}G(f,g)^{\frac p2}(x)\tilde{G}(f,g)^{\frac p2}(x)dx\\
&\le C\|G(f,g)\|^{\frac p2}_{L^p(\mathbb R)}\|\tilde G(f,g)\|^{\frac p2}_{L^p(\mathbb R)}\le C\|f\|^p_{L^{p_1}(\mathbb R)}\|g\|^p_{L^{p_2}(\mathbb R)}.
\end{align*}
This completes the proof of \eqref{var of App diff S}.
\end{proof}
\subsection{Variational weak endpoint type estimate}
To prove Proposition \ref{app var weak}, we use bilinear vector-valued Calder\'{o}n-Zygmund theory. Let $\mathcal B$ be the Banach space given by \eqref{def of B} and $F$ be a bilinear function defined on $\mathbb C\times \mathbb C$ to $\mathcal B$, we define
 $$
 \|F\|_{\mathcal{BL}(\mathbb C\times \mathbb C\rightarrow \mathcal B)}=\sup_{|\xi_1|,|\xi_2|\le1}\|F(\xi_1,\xi_2)\|_{\mathcal B}.
 $$
 \par
 Let $T$  be a bilinear operator defined on $\mathscr S(\mathbb R)\times\mathscr S(\mathbb R)$ and taking values in $\mathscr S'(\mathbb R;\mathcal B)$. Assume that the restriction of its distributional kernel away from the diagonal $x=y=z$ in $\mathbb R^3$ coincides with a $\mathcal B$-valued function $K$, satisfying the size condition
 \begin{align*}
 \|K(x,y,z)\|_{\mathcal{BL}(\mathbb C\times \mathbb C\rightarrow \mathcal B)}\le \frac C{(|x-y|+|x-z|)^2}\ \ \text{for}\ \ |x-y|+|x-z|\neq0,
 \end{align*}
 the regularity conditions
  \begin{align*}
 &\|K(x,y,z)-K(x+h,y,z)\|_{\mathcal{BL}(\mathbb C\times \mathbb C\rightarrow \mathcal B)}+\|K(x,y,z)-K(x,y+h,z)\|_{\mathcal{BL}(\mathbb C\times \mathbb C\rightarrow \mathcal B)}\\
 &+\|K(x,y,z)-K(x,y,z+h)\|_{\mathcal{BL}(\mathbb C\times \mathbb C\rightarrow \mathcal B)}\le \frac {C|h|}{(|x-y|+|x-z|)^3}
 \end{align*}
for $|h|\le\max(|x-y|,|x-z|)/2$, and such that
\begin{align*}
T(f,g)(x)=\int_{\mathbb R^2}K(x,y,z)f(y)g(z)dxdy
\end{align*}
whenever $f,g\in\mathcal D(\mathbb R)$ and $x\notin \text{supp}\ f\cap \text{supp}\ g$. Under above assumptions and $T$ is bounded $L^{p_1}\times L^{p_2}\rightarrow L^p(\mathcal B)$ for some $1<p,p_1,p_2<\infty$ with $\frac1p=\frac1{p_1}+\frac1{p_2}$, we will say $T$ is a bilinear $\mathcal B$-valued  Calder\'{o}n-Zygmund operator.
We state a weak endpoint result in \cite{JH12} for bilinear vector-valued Calder\'{o}n-Zygmund operators as follows.
\begin{lemma}\label{B weak lemma}
If $T$ is a bilinear $\mathcal B$-valued  Calder\'{o}n-Zygmund operator, then $T$ is bounded $L^1\times L^1\rightarrow L^{1/2,\infty}(\mathcal B)$.
\end{lemma}
\begin{proof} Lemma \ref{B weak lemma} implies that it suffices to verify $\{\phi_t(f,g)\}_{t>0}$ be a bilinear $\mathcal B$-valued Calder\'{o}-Zygmund operator. We have proved that $\{\phi_t(f,g)\}_{t>0}$ is bounded $L^{p_1}\times L^{p_2}\rightarrow L^p(\mathcal B)$ for $1<p,p_1,p_2<\infty$ in Proposition \ref{var of app >1}, it suffices to verify the kernel $\{\phi_t(y,z)\}_{t>0}$ satisfying related size condition and regularity conditions.
\par
 We consider the size condition first. Note that $\|a\|_{\mathcal B}=\|a\|_{V_\rho}\le \|a\|_{V_1}\le \int_0^\infty |a'(t)|dt$. Then,
\begin{align*}
\|\{\phi_t(y,z)\}_{t>0}\|_{\mathcal B}&\le \int_0^\infty\big|\frac{d}{dt}\phi_t(y,z)\big|dt\le C\int_0^\infty\big[\frac1{t^3}|\phi(\frac yt,\frac zt)|+\frac1{t^4}\big(|y||\phi_1(\frac yt,\frac zt)|+|z||\phi_2(\frac yt,\frac zt)|\big)\big]dt\\
&\le C\int_0^\infty\big[\frac1{t^3(1+\frac{|(y,z)|}t)^N}+\frac{|(y,z)|}{t^4(1+\frac{|(y,z)|}t)^N}\big]dt\le \frac{C}{|(y,z)|^2}\le \frac{C}{(|y|+|z|)^2},
\end{align*}
where $\phi_1(y,z)=\partial_y\phi(y,z)$ and $\phi_2(y,z)=\partial_z\phi(y,z)$.
\par
For the regularity condition, we have
\begin{align*}
\|\{\phi_t(y,z)-\phi_t(y',z)\}_{t>0}\|_{\mathcal B}\le C\int_0^\infty\frac{|y-y'|}{t^4(1+\frac{|(y,z)|}t)^N}dt\le \frac{C|y-y'|}{|(y,z)|^3}\le \frac{C|y-y'|}{(|y|+|z|)^3}.
\end{align*}
In the same way, we have
\begin{align*}
\|\{\phi_t(y,z)-\phi_t(y,z')\}_{t>0}\|_{\mathcal B}\le C\int_0^\infty\frac{|z-z'|}{t^4(1+\frac{|(y,z)|}t)^N}dt\le \frac{C|z-z'|}{|(y,z)|^3}\le \frac{C|z-z'|}{(|y|+|z|)^3}.
\end{align*}
This completes the proof of Proposition \ref{app var weak}.
\end{proof}

\bibliographystyle{amsplain}

\end{document}